\documentclass{article}

\RequirePackage[OT1]{fontenc}
\RequirePackage[
amsthm,amsmath,natbib]{imsart}
\usepackage{xcolor}
\usepackage[normalem]{ulem}

\usepackage{graphicx}

\usepackage{latexsym,amsmath}
\usepackage{amsmath,amsthm,amscd}
\usepackage{amsfonts}
\usepackage[psamsfonts]{amssymb}
\usepackage{pstricks,pst-plot}
\usepackage{enumerate}

\usepackage{url}

\usepackage{tocvsec2}

\usepackage{epstopdf}

\usepackage{wrapfig}

\usepackage[active]{srcltx}

\theoremstyle{plain} 
\newtheorem{theorem}{Theorem}[section]

\newtheorem{proposition}[theorem]{Proposition}

\theoremstyle{definition} 

\theoremstyle{definition} 

\newtheorem*{ex*}{Example}
\theoremstyle{remark} 

\theoremstyle{remark} 

\newtheorem*{remark*}{Remark}
\numberwithin{equation}{section}

\newcommand{\mathsym}[1]{{}}
\newcommand{\unicode}[1]{{}}

\newcommand{\iincludegraphics}[1]{{}}

\renewcommand{\le}{\leqslant}
\renewcommand{\ge}{\geqslant}

\newcommand{\ii}[1]{\operatorname{I}\left\{#1\right\}}
\newcommand{\D}{\overset{\operatorname{D}}=}
\newcommand{\dd}{\,{\mathrm d}}

\newcommand{\R}{\mathbb{R}}

\newcommand{\E}{\mathsf{E}}
\renewcommand{\P}{\mathsf{P}}

\newcommand{\al}{\alpha}
\newcommand{\be}{\beta}
\newcommand{\ga}{\gamma}
\newcommand{\ka}{\kappa}
\newcommand{\si}{\sigma}

\begin{document}

\begin{frontmatter}

\title{Exact lower bounds on the exponential moments of Winsorized and truncated random variables
}
\runtitle{Bounds on the exponential moments
}
\date{\today}

\begin{aug}
\author{\fnms{Iosif} \snm{Pinelis}\ead[label=e1]{ipinelis@mtu.edu}}
\runauthor{Iosif Pinelis}

\affiliation{Michigan Technological University}

\address{Department of Mathematical Sciences\\
Michigan Technological University\\
Houghton, Michigan 49931, USA\\
E-mail: \printead[ipinelis@mtu.edu]{e1}
}
\end{aug}

\begin{abstract}
Exact lower bounds on the exponential moments of $\min(y,X)$ and $X\ii{X<y}$ are provided given the first two moments of a random variable $X$. 
These bounds are useful in work on large deviations probabilities and nonuniform Berry-Esseen bounds, when the Cram\'er tilt transform may be employed. Asymptotic properties of these lower bounds are presented. Comparative advantages of the Winsorization $\min(y,X)$ over the truncation $X\ii{X<y}$ are demonstrated. 
\end{abstract}

 

\end{frontmatter}

\settocdepth{chapter}


\settocdepth{subsubsection}

\section{Introduction}\label{intro}

Cram\'er's tilt transform of a random variable (r.v.) $X$ is a r.v.\ $X_c$ such that 
\begin{equation}\label{eq:tilt} 
	\E f(X_c)=\frac{\E f(X)e^{c\,X}}{\E e^{c\,X}}
\end{equation}
for all nonnegative Borel functions $f$, where $c$ is a real parameter. This transform is an important tool in the theory of large deviation probabilities $\P(X>x)$, where $x>0$ is a large number; then the appropriate value of the parameter $c$ is positive.  
As e.g.\ in the proof of \cite[Theorem~2.3]{nonlin}, one often needs to bound from above the $f$-moment $\E f(X_c)$ of the $c$-tilted r.v.\ $X_c$ for a nonnegative $f$ -- and therefore one needs to bound the denominator $\E e^{c\,X}$ in \eqref{eq:tilt} from below. 
If $\E X=0$, this can be done quite easily: by Jensen's inequality, $\E e^{c\,X}\ge1$. 

A usual problem with this approach occurs when the right tail of $X$ is too heavy for $\E e^{cX}$ to be finite and hence for the transform to make sense. The standard cure in such situations is to truncate the r.v.\ $X$, say to $T_y(X):=X\ii{X\le y}$ for some real number $y>0$, where $\ii{\cdot}$ is the indicator function. Then, of course, $\E e^{c\,T_y(X)}<\infty$ for any $c>0$. However, now instead of the condition $\E X=0$ one has $\E T_y(X)\le0$, and the inequality $\E e^{c\,T_y(X)}\ge1$ (in place of $\E e^{c\,X}\ge1$) will not hold in general. In fact, $\E e^{c\,T_y(X)}$ can be however small for some $c>0$, even if one imposes a restriction such as $\E X^2\le\si^2$ for a given real $\si>0$ -- see the discussion in Subsection~\ref{win/trunc}. 

A much better way to cut off the right tail of the distribution of $X$ is the so-called Winsorization. That is, instead of the truncation $T_y(X)$, one 
deals with $W_y(X):=y\wedge X=\min(y,X)$. 
Clearly, $W_y(X)\ge T_y(X)$ and hence $\E e^{c\,W_y(X)}\ge\E e^{c\,T_y(X)}$ for $c>0$. 
Moreover, it turns out that for any given real $\si>0$ and $y>0$ the infimum of $\E e^{c\,W_y(X)}$ over all $c>0$ and all r.v.'s $X$ with $\E X\ge0$ and $\E X^2\le\si^2$ is strictly positive; furthermore, it decreases slowly from $1$ to $0$ as $\si$ increases from $0$ to $\infty$. 
These properties of Winsorization make it a clear winner over truncation in many relevant situations.

\section{Results}\label{results}
Take any real $\si>0$. 
Let $X$ denote any r.v.\ with 
\begin{equation*}
	\E X\ge0\quad\text{and}\quad \E X^2\le\si^2.
\end{equation*}

For any positive real $a$ and $b$, let $X_{a,b}$ stand for any zero-mean r.v.\ with values in the two-point set $\{-a,b\}$; thus, the distribution of $X_{a,b}$ is uniquely determined by $a$ and $b$. 
Note also that $\E X_{a,b}^2=ab$. 

\subsection{Winsorization}\label{win}

Consider the Winsorization function defined by the formula  
\begin{equation}\label{eq:W}
	W(x):=1\wedge x. 
\end{equation}

The following proposition allows one to define the terms in which to express the exact lower bounds on $\E e^{c\,W(X)}$. 

\begin{proposition}\label{prop:win} 
Take any real $c>0$. 
\begin{enumerate}[(I)]
	\item For any real $a$, let 
\begin{equation}\label{eq:b*_a,c}
	b^*_{a,c}:=\tfrac{2(e^{c+ac}-1)-ac}c. 
\end{equation}
Then the equation $a\,b^*_{a,c}=\si^2$ has a unique positive root, say $a_{c,\si}$, so that 
\begin{equation}\label{eq:a_c,si}
	\{a_{c,\si}\}=\big\{a>0\colon a\,b^*_{a,c}=\si^2\big\}. 
\end{equation}
\item The expression 
\begin{equation}\label{eq:l1}
	\ell_1(a):=\ell_1(a,\si):=
	\ln\tfrac{a}{\si^2}-\tfrac{2(a+1)(a-\si^2)}{a^2+\si^2} 
\end{equation}
switches in sign exactly once, from $-$ to $+$, as $a$ increases from $0$ to $\si^2$. 
Therefore,
one can uniquely define $a_\si$ by the formula 
\begin{equation}\label{eq:a_si}
	\{a_\si\}=\big\{a\in(0,\si^2)\colon\ell_1(a)=0\big\}. 
\end{equation}
\end{enumerate}
\end{proposition}

The proofs are deferred to Section~\ref{proofs}. 

Now we are ready to define three more symbols: 
\begin{equation}\label{eq:b_c,si}
	b_{c,\si}:=\si^2/a_{c,\si}; 
\end{equation}
\begin{equation}\label{eq:b_si}
	b_\si:=\si^2/a_\si; 
\end{equation}
\begin{equation}\label{eq:c_si}
	c_\si:=\tfrac{\ln b_\si}{1+a_\si}. 
\end{equation}

\begin{theorem}\label{th:win}
For any real $c>0$ 
\begin{align}
	\E\exp\{c\,W(X)\}&\ge L_{W;c,\si}:=\E\exp\{c\,W(X_{a_{c,\si},b_{c,\si}})\} 
	\label{eq:c fixed}\\
	&\ge L_{W;\si}:=\E\exp\big\{c_\si\,W(X_{a_\si,b_\si})\big\}. 
	\label{eq:c any}
\end{align}
Moreover, inequality \eqref{eq:c fixed} is strict unless $X\D X_{a_{c,\si},b_{c,\si}}$, where $\D$ denotes the equality in distribution, and inequality \eqref{eq:c any} is strict unless $c=c_\si$. 
Furthermore, 
\begin{equation}\label{eq:b_si=}
	a_\si=a_{c_\si,\si}\quad\text{and}\quad 
	b_\si=b_{c_\si,\si}=b^*_{a_\si,c_\si},
\end{equation}  
so that \eqref{eq:c any} turns into equality if and only if $c=c_\si$. 
\end{theorem}

In addition to being zero-mean, each of the r.v.'s $X_{a_{c,\si},b_{c,\si}}$ and $X_{a_\si,b_\si}$ has variance $\si^2$, in view of \eqref{eq:b_c,si} and \eqref{eq:b_si}.  
Moreover, by \eqref{eq:a_si} and \eqref{eq:b_si}, $b_\si>1$ and hence $c_\si>0$ by \eqref{eq:c_si}. 
Thus, \eqref{eq:c fixed} provides an exact lower bound on $\E\exp\{c\,W(X)\}$ for a fixed $c>0$, while \eqref{eq:c any} provides an exact lower bound on $\E\exp\{c\,W(X)\}$ over all $c>0$. 

Let us now describe the asymptotics of the bounds $L_{W;c,\si}$ and $L_{W;\si}$ for $\si\downarrow0$ and $\si\to\infty$. As usual, we write $a\sim b$ if $\frac ab\to1$. 

\begin{proposition}\label{prop:win asymp} \ 

\begin{enumerate}[(I)]
	\item 
For any real $c>0$  
\begin{alignat}{2}
	L_{W;c,\si}-1&\sim\tfrac{-c^2}{4(e^c-1)}\,\si^2\,\ &&\text{as }\si\downarrow0, \label{eq:W,c;si->0} \\
	L_{W;c,\si}&\sim\tfrac{4e^c}{c^2}\,\tfrac{\ln^2\si}{\si^2}\ &&\text{as }\si\to\infty. \label{eq:W,c;si->infty} 
\end{alignat}
\item 
The expression 
\begin{equation}\label{eq:f}
	f(t):=\ln t +2 (1 - t) 
\end{equation}
switches in sign exactly once, from $-$ to $+$, as $t$ increases from $0$ to $1$;  
Therefore,
one can uniquely define $t_*$ by the formula 
\begin{equation}\label{eq:t}
	\{t_*\}=\big\{t\in(0,1)\colon f(t)=0\big\};  
\end{equation}
in fact, $t_*=0.203\dots$. 
\item
\begin{alignat}{2}
	L_{W;\si}-1&\sim-(1-t_*)t_*\,\si^2\,\ &&\text{as }\si\downarrow0, \label{eq:W;si->0} \\
	L_{W;\si}&\sim e^2\,\tfrac{\ln^2\si}{\si^2}\ &&\text{as }\si\to\infty, \label{eq:W;si->infty} 
\end{alignat}
\item Comparing \eqref{eq:W,c;si->0} with \eqref{eq:W;si->0}, and \eqref{eq:W,c;si->infty} with \eqref{eq:W;si->infty}: 
\begin{alignat}{2}
	\inf_{c>0}\tfrac{-c^2}{4(e^c-1)}&=-(1-t_*)t_*
	\ &&\text{-- attained at }c=-\ln t_*=1.593\dots; \label{eq:W,sup,si->0} \\
	\inf_{c>0}\tfrac{4e^c}{c^2}&=e^2\ &&\text{-- attained at }c=2. \label{eq:W,sup,si->infty}
\end{alignat}
Thus, the asymptotic expression in \eqref{eq:W;si->0} is the minimum in $c>0$ of that in \eqref{eq:W,c;si->0}, and the asymptotic expression in \eqref{eq:W;si->infty} is the minimum in $c>0$ of that in \eqref{eq:W,c;si->infty}. 
\end{enumerate}
\end{proposition}

Note that the convergence for $\si\to\infty$ in Proposition~\ref{prop:win asymp} is very slow. E.g., the ratio $e^2\,\tfrac{\ln^2\si}{\si^2}/L_{W;\si}$ of the terms in \eqref{eq:W;si->infty} is $1.201\dots$ for 
$\si$ as large as $10^{10}$. 

The relations $a_\si\sim t_*\,\si^2$ as $\si\downarrow0$ and $a_\si\sim \frac12\ln(\si^2)$ as $\si\to\infty$, to be  established in the 
proof of part (III) of Proposition~\ref{prop:win asymp}, suggest, and numerical calculations confirm, that a good initial approximation for solving the equation $\ell_1(a)=0$ in \eqref{eq:a_si} for $a_\si$ is $a=\frac12\ln(1 + 2t_*\,\si^2)\approx\frac12\ln(1 + 0.406\,\si^2)$. 

\subsection{Truncation}\label{trunc}
Consider the truncation function $T$ defined by the formula 
\begin{equation}\label{eq:T}
	T(x):=x\ii{x<1}. 
\end{equation}

The following proposition allows one to define the terms in which to express the exact lower bounds on $\E e^{c\,T(X)}$. 

\begin{proposition}\label{prop:trunc} 
Take any real $c>0$. 
For any real $a$, let 
\begin{equation*}\label{eq:B*_a,c}
	B^*_{a,c}:=\tfrac{2(e^{ac}-1)-ac}c;  
\end{equation*}
cf.\ \eqref{eq:b*_a,c}. 
Then one can uniquely define $A_\si$ and $A_{c,\si}$ by the formulas 
\begin{align}
	\{A_c\}&=\big\{a>0\colon B^*_{a,c}=1\big\},  \label{eq:A_c}\\
	\{A_{c,\si}\}&=\big\{a>0\colon a\,B^*_{a,c}=\si^2\big\}, \label{eq:A_c,si}
\end{align}
because each of the equations on the right-hand sides of \eqref{eq:A_c} and \eqref{eq:A_c,si} has a unique root $a>0$. 
Moreover, one has the implication 
\begin{equation}\label{eq:A>A}
A_c\le\si^2\implies	A_{c,\si}\ge A_c.
\end{equation}
\end{proposition}

We shall need one more definition: 
\begin{equation*}\label{eq:B_c,si}
	B_{c,\si}:=\si^2/A_{c,\si}.  
\end{equation*}

\begin{theorem}\label{th:trunc}
For any real $c>0$ 
\begin{equation}\label{eq:trunc}
	\E\exp\{c\,T(X)\}\ge
	L_{T;c,\si}:=
	\left\{
\begin{alignedat}{2}	
	\E\exp\{c\,T\big(X_{\si^2,1}\big)\} 						&
	&&\text{ if }\si^2\le A_c, \\
	\E\exp\{c\,T\big(X_{A_{c,\si},B_{c,\si}}\big)\} &
&&\text{ if }\si^2\ge A_c. 
\end{alignedat}
\right. 
\end{equation}
Moreover, inequality \eqref{eq:trunc} is strict unless $X$ equals $X_{\si^2,1}$ or $X_{A_{c,\si},B_{c,\si}}$ in distribution, depending on whether $\si^2\le A_c$ or $\si^2\ge A_c$. 
\end{theorem}

To complete this subsection, let us describe the asymptotics of the bound $L_{T;c,\si}$ for $\si\downarrow0$ and $\si\to\infty$ -- cf.\ Proposition~\ref{prop:win asymp}. 

\begin{proposition}\label{prop:trunc asymp} 
For any real $c>0$
\begin{alignat}{2}
	L_{T;c,\si}-1&\sim-c\,\si^2\,\ &&\text{as }\si\downarrow0, \label{eq:T,c;si->0} \\
	L_{T;c,\si}&\sim\tfrac4{c^2}\,\tfrac{\ln^2\si}{\si^2}\ &&\text{as }\si\to\infty. \label{eq:T,c;si->infty}
\end{alignat}
\end{proposition}

\subsection{Winsorization and truncation: discussion and comparison}\label{win/trunc}

The Winsorization function $W$  
and the truncation function $T$ as defined by \eqref{eq:W} and \eqref{eq:T} ``cut'' a given value $x$ at the level $1$. However, by simple rescaling it is easy to restate the results for any positive ``cut'' level $y$. Indeed, one may consider $W_y(x):=y\wedge x$ and $T_y(x):=x\ii{x<y}$, so that $W_1=W$ and $T_1=T$. Then $c\,W_y(X)=c\,y\,W(X/y)$ and $c\,T_y(X)=c\,y\,T(X/y)$. 
Now one can use the results of Subsections~\ref{win} and \ref{trunc} with $c$, $X$, and $\si$ replaced by $c\,y$, $X/y$, and $\si/y$, respectively. 
It should therefore be clear that the ``cut'' level was set to be $1$ just for the simplicity of presentation. 

Observe that for each $c>0$ the exact lower bound $L_{W;c,\si}$ in \eqref{eq:c fixed} is no greater than $1$, since the zero r.v.\ $X$ obviously satisfies the conditions $\E X\ge0$ and $\E X^2\le\si^2$. Hence, the exact lower bounds $L_{W;\si}$ and $L_{T;c,\si}$, which are no greater than $L_{W;c,\si}$, are as well no greater than $1$. It is also clear that each of these exact lower bounds is nondecreasing in $\si$ -- since the exactness is over all r.v.'s $X$ with $\E X\ge0$ and $\E X^2\le\si^2$. 


\begin{wrapfigure}{l}{0.50\textwidth}
\includegraphics[width=0.52\textwidth]{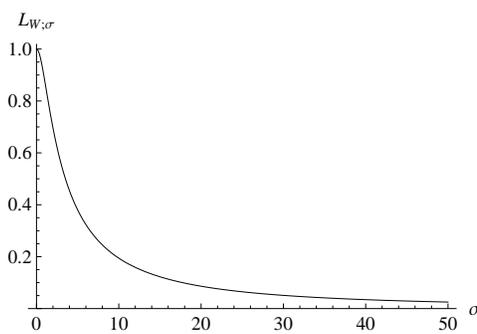}
	\caption{The exact lower bound $L_{W;\si}$.}
	\label{fig:c any}
	\vspace*{-10pt}
\end{wrapfigure}



However, for any $c>0$ the exact lower bound $L_{W;c,\si}$ for the Winsorized r.v.\ $W(X)$ decreases rather slowly from $1$ to $0$ as $\si$ increases from $0$ to $\infty$. Even the smaller, universal over all $c>0$ exact lower bound $L_{W;\si}$ decreases rather slowly; see Figure~\ref{fig:c any} and also recall Proposition~\ref{prop:win asymp}. 
In particular, for the value $\si^2=1$ (which is of special interest as far as the application in \cite{nonlin} is concerned) the lower bound $L_{W;\si}$ is $0.878\dots$, rather close to $1$. 
Even for $\si^2=100$, this bound is $0.194\dots$, not very small. 

Moreover, the universal (over all $c>0$) Winsorization bound $L_{W;\si}$ is remarkably close to the fixed-$c$ Winsorization bounds $L_{W;c,\si}$, especially if the value of $c$ is in the interval $[1,3]=\{\frac p2\colon 2\le p\le 6\}$ -- which is of particular interest in \cite{nonlin}. See the picture at the top of  Figure~\ref{fig:ratios}; the green graph there, for $c=2$ -- cf.\ \eqref{eq:W,sup,si->infty} -- looks exactly horizontal at level $1$, but it is in fact not.  

\begin{wrapfigure}{l}{0.48\textwidth}
\includegraphics[width=0.50\textwidth]{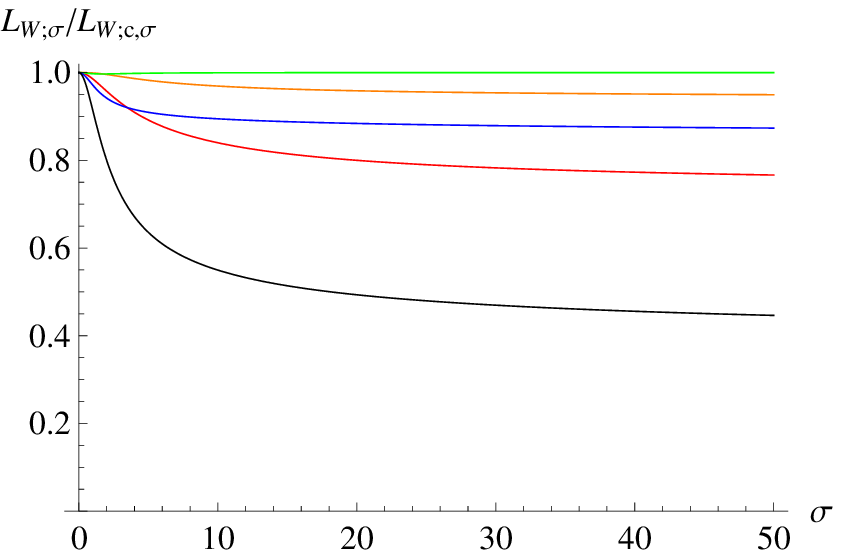}

\vspace*{8pt}
\includegraphics[width=0.50\textwidth]{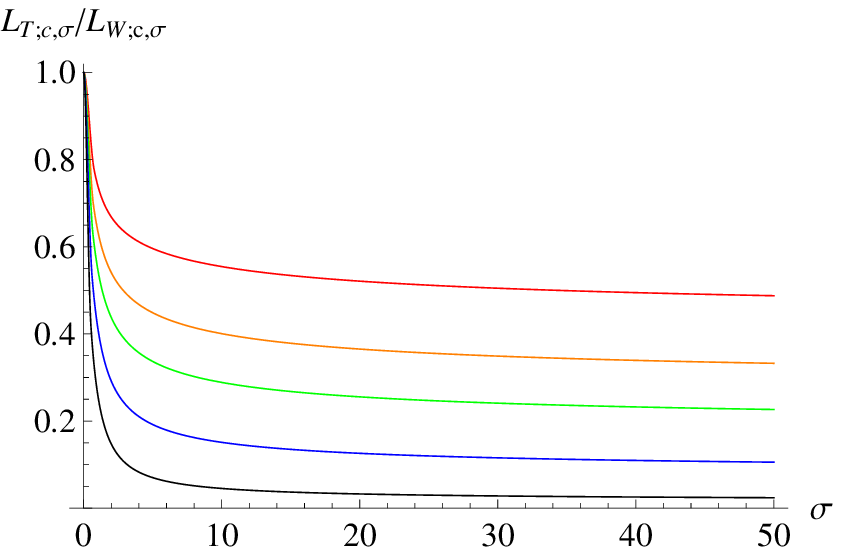}
	\caption{The ratios $L_{W;\si}/L_{W;c,\si}$ (above) and $L_{T;c,\si}/L_{W;c,\si}$ (below) for $c=1,\frac32,2,3,5$ -- red, orange, green, blue, black, respectively.}
	\label{fig:ratios}
		\vspace*{-5pt}
\end{wrapfigure}


As for the truncation case, it is quite different from the Winsorization one. Indeed, the exact lower bound $L_{T;c,\si}$ is significantly smaller than $L_{W;c,\si}$, especially for larger values of $\si$. The bottom picture of Figure~\ref{fig:ratios} 
shows the graphs of the ratios of these two bounds for $c=1,\frac32,2,3,5$. 

It is also easy to compare the asymptotics for $L_{T;c,\si}$ in \eqref{eq:T,c;si->0} and \eqref{eq:T,c;si->infty} with that for $L_{W;c,\si}$ in \eqref{eq:W,c;si->0} and \eqref{eq:W,c;si->infty}. 
Comparing \eqref{eq:W,c;si->0} with \eqref{eq:T,c;si->0}, it is easy to see that for $\si\downarrow0$ one has $1-L_{T;c,\si}$ is at least $4$ times as large (asymptotically) as $1-L_{W;c,\si}$, and may be infinitely many times as large when $c$ goes to $0$ or $\infty$. Similarly, for $\si\to\infty$, $L_{W;c,\si}$ is $e^c$ times as large (asymptotically) as $L_{T;c,\si}$.


Moreover, in contrast with the Winsorization case, there is no nontrivial lower bound in the truncation case that would be universal over all $c>0$. Namely, for any given $\si>0$, the infimum of $\E\exp\{c\,T(X)\}$ over all $c>0$ and all r.v.'s $X$ with $\E X\ge0$ and $\E X^2\le\si^2$ is $0$; the same holds even if the conditions $\E X\ge0$ and $\E X^2\le\si^2$ are strengthened to $\E X=0$ and $\E X^2=\si^2$. 
Indeed, let $a\downarrow0$, $b:=\si^2/a$, and $c:=1/a^2$; then it is easy to see that $\E\exp\{c\,T(X_{a,b})\}\to0$; cf.\ \eqref{eq:T,c;si->0} and \eqref{eq:T,c;si->infty}, with large $c$. 

The general problem of finding the maximum or minimum of the generalized moment $\int f\dd\mu$ over the set of all nonnegative measures with given generalized moments $\int f_i\dd\mu$ ($i\in I$) 
goes back to Chebyshev and Markov; here a function $f$ and a family of functions $(f_i)_{i\in I}$ are given; see, e.g., \cite{krein-nudelman,karlin-studden,hoeff-extr,kemper-dual,karr,pin98}. One group of results in this area is that for finite $I$ under general conditions it may be assumed without loss of generality that the support of $\mu$ is also finite, with cardinality no greater than that of $I$; methods based on such results may be referred to as finite-support methods. Other results, valid for finite or infinite $I$, concern the following duality: under general conditions, 
the supremum (say) of $\int f\dd\mu$ over all $\mu$ such that $\int f_i\dd\mu\le c_i$ for all $i\in I$ coincides with the infimum of $\int_I c_i\,\nu(\!\dd i)$ over all nonnegative measures $\nu$ on $I$ (say with a finite support) such that $\int_I f_i\,\nu(\!\dd i)\ge f$. 

Such methods were used e.g.\ in 
\cite{bennett,hoeff63,pin-utev-exp,utev-extr,bent-liet02,bent-hoeff,pin-hoeff,tyurin}. In particular, the supremum of $\int_{\R} e^x\,\mu(\!\dd x)$ given $\int_y^\infty\mu(\!\dd x)=0$, $\int_{\R} x\,\mu(\!\dd x)=0$, and $\int_{\R} x^2\,\mu(\!\dd x)=\si^2$ was found (implicitly) in \cite{bennett} and (explicitly) in \cite{pin-utev-exp}. 

In \cite{hoeff63} a similar problem was solved, under the additional restriction that $\mu$ is a probability measure. This result was extended in \cite{bent-liet02,bent-hoeff} to the Eaton-like moment functions $(\cdot-t)_+^2$ ($t\in\R)$ in place of the exponential function $e^\cdot$ in \cite{hoeff63}; on the other hand, this was a further development of the line of results obtained in \cite{eaton1,eaton2,pin98,pin99}. 
The results of \cite{bennett,hoeff63} and \cite{bent-liet02,bent-hoeff} were refined in \cite{pin-utev-exp} and \cite{pin-hoeff}, respectively, by also taking into account positive-part third moments. 
The supremum of the moments $\int f\dd\mu$ over all Stein-type moment functions $x\mapsto f(x):=xg(x)-g'(x)$ with Lipschitz-1 functions $g'$ and over all probability measures $\mu$ with 
given mean, variance, and third absolute moment  
was presented (in an equivalent form) in \cite[Theorem~3]{tyurin}. 
%
Results somewhat related to the mentioned ones were obtained in \cite{disintegr}; see also the bibliography therein. 
Of course, mentioned above are a very small sample of the work done on the Chebyshev-Markov type of extremal problems. 

Concerning our problems of minimizing the exponential moments of $W(X)$ and $T(X)$, 
one could use mentioned finite-support methods to reduce the consideration to r.v.'s $X$ taking only three values, since we have here three affine restrictions: on the first two moments and on the total mass of the measure (which has to be a probability measure). Another, more ad hoc kind of approach would be to condition the distribution of the r.v.\ $X$ on $\ii{X<1}$, which would preserve the mean and would not increase the second moment; also, this conditioning would not increase the exponential moments of $W(X)$ and $T(X)$, since both functions, $e^{cW}$ and $e^{cT}$, are convex on $(-\infty,1)$ and on $[1,\infty)$; thus, it would remain to consider r.v.'s $X$ taking only two values. 
However, the duality-type method that we chose to prove (in the next section) inequalities \eqref{eq:c fixed} and \eqref{eq:trunc} appears more effective, as it immediately reduces the consideration to r.v.'s $X_{a,b}$ that, not only take just two values, but also have the first two moments exactly equal to $0$ and $\si^2$, respectively; moreover, this approach appears more convenient in obtaining the strictness conditions for inequalities \eqref{eq:c fixed} and \eqref{eq:trunc}. 

%
%
%
%
%

\section{Proofs}\label{proofs}

\begin{proof}[Proof of Proposition~\ref{prop:win}]\ 

(I).\quad Part (I) follows because $ab^*_{a,c}$ strictly and continuously increases from $0$ to $\infty$ as $a$ increases from $0$ to $\infty$. 

(II).\quad 
Observe that $a (a^2 + \si^2)^2 \ell_1'(a)$ is a quadratic polynomial in $\si^2$, whence one can see that the system of inequalities $\ell_1'(a)>0$ and $0<a<\si^2$ can be rewritten as $0 < a < \sqrt{1 + \si^2}-1$. 
This means that $\ell_1(a)$ switches from increase to decrease over $a\in(0,\si^2)$; at that, $\ell_1(0+)=-\infty$ and $\ell_1(\si^2)=0$. Now part (II) of Proposition~\ref{prop:win} follows as well. 
\end{proof}

\begin{proof}[Proof of Theorem~\ref{th:win}]
Let 
\begin{equation}\label{eq:F,G}
F(x):=e^{c\,W(x)}\quad\text{and}\quad G(x):=\al+\be x+\ga x^2	
\end{equation}
for all $x\in\R$, 
where 
\begin{equation}\label{eq:al,be,ga}
	\al:=e^c-\frac{b^2 c e^{-a c}}{2 (a+b)},\quad 
\be:=\frac{b c e^{-a c}}{a+b},\quad
\ga:=-\frac{c e^{-a c}}{2 (a+b)}, 
\end{equation}
$a>0$, $b>1$, $c>0$. 
Then it is straightforward to check that 
$F(b)=G(b)$, $F'(b)=G'(b)$, and $F'(-a)=G'(-a)$.  
Let now $b=b^*_{a,c}$. Then
$b>2>1$ and $F(-a)=G(-a)$, so that 
\begin{equation}
	F(b)=G(b),\quad F'(b)=G'(b),\quad F(-a)=G(-a),\quad
	F'(-a)=G'(-a).   
\end{equation} 

Also, by \eqref{eq:al,be,ga}, $\ga<0$ and hence the function $G$ is strictly concave, while the function $F$ is convex on
$(-\infty,1)$ and on $[1,\infty)$; so, the difference $D:=F-G$ is strictly convex on $(-\infty,1)$ and on $[1,\infty)$; at that, $D(-a)=D'(-a)=D(b)=D'(b)=0$, whence   
$D>0$ and $F>G$ on $\R\setminus\{-a,b\}$, while $F=G$ on the two-point set $\{-a,b\}$.   
Now specify $a$, to $a=a_{c,\si}$. Then, recalling \eqref{eq:b_c,si} and \eqref{eq:a_c,si}, one sees that 
\begin{equation}\label{eq:b=b}
	b^*_{a_{c,\si},c}=b_{c,\si}.  
\end{equation}
Therefore, also specifying $b$ to $b=b_{c,\si}$, one has 
\begin{equation}\label{eq:ge}
\E e^{c\,W(X)}=\E F(X)\ge\E G(X)\ge\E G(X_{a,b})=\E F(X_{a,b})
=\E\exp\big\{c\,W(X_{a,b})\big\}; 
\end{equation}
the second inequality here takes place because \big(in view of \eqref{eq:al,be,ga}\big) $\be>0>\ga$, while $\E X\ge0=\E X_{a,b}$ and $\E X^2\le\si^2=ab=\E X_{a,b}$.  
Thus, \eqref{eq:c fixed} follows. 
Moreover, because 
$F>G$ on $\R\setminus\{-a,b\}$, the first inequality in \eqref{eq:ge} is strict unless the support of the distribution of $X$ is a subset of $\{-a,b\}$, and the second inequality in \eqref{eq:ge} is strict unless $\E X=0$ and $\E X^2=\si^2=ab$; thus indeed, inequality \eqref{eq:c fixed} is strict unless $X\D X_{a,b}$. 

To prove inequality \eqref{eq:c any}, it suffices to show that $\E\exp\big\{c_\si\,W(X_{a_\si,b_\si})\big\}$ is a lower bound on $\E\exp\big\{c\,W(X_{a,b})\big\}$ for any $a$ and $b$ such that $a\in(0,\si^2)$ and $b=\si^2/a$. 
Take indeed any such $a$ and $b$. Observe that $\E\exp\big\{c\,W(X_{a,b})\big\}
=\frac{a^2 e^c+e^{-a c} \si^2}{a^2+\si^2}$ is strictly convex in $c\in\R$ and attains its minimum,  
\begin{equation*}
	m(a,\si):=\frac{a (1+a) \left(\frac{a}{\si^2}\right)^{-\frac{1}{1+a}}}{a^2+\si^2},
\end{equation*}
in $c$ only at $c=\frac{\ln(\si^2/a)}{1+a}=\frac{\ln b}{1+a}$ -- cf.\ \eqref{eq:c_si}. 
Next, the minimum of $m(a,\si)$ or, equivalently, of 
\begin{equation*}
	\ell(a):=\ell(a,\si):=\ln m(a,\si)
\end{equation*}
in $a\in(0,\si^2)$ is attained only at the point $a=a_\si$ defined by \eqref{eq:a_si}, because,  
by part (II) of Proposition~\ref{prop:win}, $\ell'(a)=\ell_1(a)/(1+a)^2$ 
switches in sign from $-$ to $+$ over $a\in(0,\si^2)$. 
Thus, 
inequality \eqref{eq:c any} is true, and it is strict unless $c=c_\si$ and $a_{c,\si}=a_\si$. 

It remains to verify the three equalities in \eqref{eq:b_si=}. 
In view of \eqref{eq:b=b}, the last of these equalities is implied by the first one. 
So, if any of the equalities in \eqref{eq:b_si=} were false, then the equality $X_{a_\si,b_\si}\D X_{a_{c_\si,\si},b_{c_\si,\si}}$ would also be false, and so, by what has been proved, inequality \eqref{eq:c fixed} with $c_\si$ and $X_{a_\si,b_\si}$ in place of $c$ and $X$ would be strict, which would contradict inequality \eqref{eq:c any}. 

This completes the proof of Theorem~\ref{th:win}.    
\end{proof}


\begin{proof}[Proof of Proposition~\ref{prop:win asymp}]\ 

(I)\quad To prove part (I), consider first the case $\si\downarrow0$. Then, by \eqref{eq:a_c,si} and \eqref{eq:b*_a,c}, $a_{c,\si}\downarrow0$. Moreover, $b^*_{a,c}\to\frac{2(e^c-1)}c$ whenever $a\to0$. So, by \eqref{eq:b=b}, $b_{c,\si}\to\frac{2(e^c-1)}c>2>1$, and so, $a_{c,\si}=\si^2/b_{c,\si}\sim\frac c{2(e^c-1)}\,\si^2$. 
On the other hand, 
\begin{align}\label{eq:E W}
	\E\exp\{c\,W(X_{a,b})\}-1=(e^c-1)\,\tfrac a{a+b}+(e^{-ca}-1)\,\tfrac b{a+b}&\sim(\tfrac{e^c-1}b-c)a 
\end{align}
whenever 
$a\downarrow0$ and $a=o(b)$. 
This, together with the relations $b_{c,\si}\to\frac{2(e^c-1)}c$ and $a_{c,\si}\sim\frac c{2(e^c-1)}\,\si^2$, implies 
\eqref{eq:W,c;si->0}. 

Consider now the case $\si\to\infty$. Then, by \eqref{eq:a_c,si} and \eqref{eq:b*_a,c}, $a_{c,\si}\to\infty$. Next, 
\begin{equation}\label{eq:b*}
b^*_{a,c}\sim\tfrac{2e^c}c\,e^{ac}\quad\text{as }a\to\infty. 	
\end{equation}
So, in view of \eqref{eq:b_c,si} and \eqref{eq:b=b}, 
for $a=a_{c,\si}$ one has 
\begin{equation}\label{eq:si^2}
\si^2\sim\tfrac{2e^c}{c^2}\,ac\,e^{ac}=e^{ac\,(1+o(1))},	
\end{equation}
whence $a_{c,\si}=a\sim\frac1c\,\ln(\si^2)$ and $b_{c,\si}=\si^2/a_{c,\si}\sim c\si^2/\ln(\si^2)$. 
Also, for $a\to\infty$ and $b=b^*_{a,c}$ \eqref{eq:b*} yields $a=o(b)$ and 
\begin{equation}\label{eq:E W infty}
	\E\exp\{c\,W(X_{a,b})\}=\tfrac{ae^c+be^{-ac}}{a+b}
	\sim\tfrac{ae^c+2e^c/c}{a+b}
	\sim\tfrac{ae^c}b 
\end{equation} 
This, together with the relations $a_{c,\si}\sim\frac1c\,\ln(\si^2)$ and $b_{c,\si}\sim c\si^2/\ln(\si^2)$, implies \eqref{eq:W,c;si->infty}. 

(II)\quad Note that $f(0+)=-\infty$, $f(1)=0$, and $f'(t)=-2(t-\frac12)/t$ switches in sign from $+$ to $-$ as $t$ increases from $0$ to $1$. Now part (II) of Proposition~\ref{prop:win asymp} follows. 

(III)\quad To prove part (III), consider first the case $\si\downarrow0$. Then, by \eqref{eq:l1} and \eqref{eq:f}, for each fixed $t\in(0,1)$ one has 
$\ell_1(t\si^2)=\ln t-(2+o(1))(t-1)=f(t)+o(1)$. So, by part (II) of Proposition~\ref{prop:win asymp} for each fixed  $t\in(0,t_*)$ one has $\ell_1(t\si^2)<0$ -- eventually, for all small enough $\si$; similarly, for each fixed  $t\in(t_*,1)$ eventually $\ell_1(t\si^2)>0$. Therefore, by \eqref{eq:a_si}, 
$a_\si\sim t_*\si^2$ and hence, by \eqref{eq:b_si} and \eqref{eq:c_si}, $b_\si\to1/t_*$ and $c_\si\to-\ln t_*$. 
Now \eqref{eq:W;si->0} follows by \eqref{eq:E W}, since for $c=c_\si$ one has $e^c\to1/t_*$ and $c\to-\ln t_*=2(1-t_*)$, the last equality due to \eqref{eq:t}-\eqref{eq:f}. 

The case $\si\to\infty$ is considered similarly. Then, by \eqref{eq:l1}, $\ell_1\big(\ka\,\ln(\si^2)\big)\sim 2(\ka-\frac12)\,\ln(\si^2)$ for each fixed $\ka\in(0,\infty)\setminus\{\frac12\}$, so that $\ell_1\big(\ka\,\ln(\si^2)\big)$ is eventually less than $0$ for each $\ka\in(0,\frac12)$ and eventually greater than $0$ for each $\ka\in(\frac12,\infty)$. Thus, by part (II) of Proposition~\ref{prop:win}, $a_\si\sim\frac12\,\ln(\si^2)$ and hence $b_\si\sim2\si^2/\ln(\si^2)$ and $c_\si\to2$. 
Moreover, by \eqref{eq:b_si=}, one has $b_\si=b^*_{a_\si,c_\si}$. 
Recall that relations \eqref{eq:b*} and \eqref{eq:E W infty} were derived assuming that $a\to\infty$, $b=b^*_{a,c}$, and $c>0$ is fixed. Reasoning quite similarly -- with $a_\si$, $b_\si$, $c_si$ in place of such $a$, $b$, $c$ -- one concludes that $L_{W;\si}\sim a_\si\,e^{c_\si}/b_\si$, and now 
\eqref{eq:W;si->infty} follows since $a_\si\sim\frac12\,\ln(\si^2)$, $b_\si\sim2\si^2/\ln(\si^2)$, and $c_\si\to2$. 

(IV)\quad The derivative of $\tfrac{-c^2}{4(e^c-1)}$ in $c>0$ is positive iff $f(t)<0$ for $t:=e^{-c}$. Therefore and by part (II) of Proposition~\ref{prop:win asymp}, $\tfrac{-c^2}{4(e^c-1)}$ attains a minimum in $c>0$ at $c=-\ln t_*=2(1-t_*)$. Replacing now $c$ in the denominator of $\tfrac{-c^2}{4(e^c-1)}$ by $-\ln t_*$ and in the numerator, by $2(1-t_*)$, one obtains \eqref{eq:W,sup,si->0}. As for \eqref{eq:W,sup,si->infty}, it is much more straightforward. 
\end{proof}

\begin{proof}[Proof of Proposition~\ref{prop:trunc}]\ 
That each of the equations on the right-hand sides of \eqref{eq:A_c} and \eqref{eq:A_c,si} has a unique root $a>0$ follows because both $B^*_{a,c}$ and $aB^*_{a,c}$ strictly and continuously increase from $0$ to $\infty$ as does so. 
Now \eqref{eq:A>A} follows because the value of $a\,B^*_{a,c}$ at $a=A_c$ is $A_c$. 
\end{proof}

\begin{proof}[Proof of Theorem~\ref{th:trunc}]
Let here $F(x):=e^{c\,T(x)}$ and let $G(x)$ be defined as in \eqref{eq:F,G}. 

Consider first the case $\si^2\le A_c$. Here, take $G(x)$ with 
\begin{gather*}
	\al:=\frac{e^{-ac}(a^2 e^{ac}+ca^2+ac+2a+1)}{(a+1)^2},\\
	\be:=\frac{e^{-ac} \big(2 a(e^{ac}-1)+c(1-a^2)\big)}{(a+1)^2},\quad
	\ga:=\frac{e^{-ac} (e^{ac}-ac-c-1)}{(a+1)^2}. 
\end{gather*}
Then for $D:=F-G$ and any $a>0$ one has $D(-a)=D'(-a)=D(1)=0$. 
Let now $a=\si^2$, so that (by the current case condition) $0<a\le A_c$, which implies $B^*_{a,c}\le1$ and hence $G'(1)\le0$ \big(because $G'(1)=(B^*_{a,c}-1)c\,e^{-ac}/(1+a)$\big); also, $0<a\le A_c$ implies $(e^{ac}-1)-(1+a)c<2(e^{ac}-1)-(1+a)c=(B^*_{a,c}-1)c\le0$, 
whence $\ga<0$, so that $G$ is strictly convex on $\R$; moreover, $\be>e^{-ac}c\,(1+a^2)/(1+a)^2>0$. 
In turn, the inequality $G'(1)\le0$ means that $D'(1+)\ge0$; also, $D$ is strictly convex on $(-\infty,1)$ and $[1,\infty)$; recalling now that $D(-a)=D'(-a)=D(1)=0$, one has $D>0$ and $F>G$ on $\R\setminus\{-a,1\}$, while $F=G$ on the two-point set $\{-a,1\}=\{-\si^2,1\}$. 
Now the first line of \eqref{eq:trunc} follows -- cf.\ \eqref{eq:ge}. 

Consider now the case $\si^2\ge A_c$. Here, take $G(x)$ with 
\begin{align*}
	\al:=\frac{e^{-ac}(ca^2+2abc+2a+2b)}{2(a+b)},\quad 
	\be:=\frac{cbe^{-ac}}{a+b},\quad
	\ga:=-\frac{c e^{-ac}}{2(a+b)},
\end{align*}
where $a>0$ and $b:=B^*_{a,c}$. 
Assume now also that $a$ is so large as $b\ge1$. 
Then, again for $D:=F-G$, one has $D(-a)=D'(-a)=D(b)=D'(b)=0$, while $\be>0>\ga$, so that again $D$ is strictly convex on $(-\infty,1)$ and $[1,\infty)$, $D>0$ and $F>G$ on $\R\setminus\{-a,b\}$, while $F=G$ on the two-point set $\{-a,b\}$; if $b=1$ then $D'(b)$ is understood as the right derivative of $D$ at point $1$. 
Since the current case if $\si^2\ge A_c$, \eqref{eq:A>A} yields 
$A_{c,\si}\ge A_c$ and hence 
$B_{c,\si}=\si^2/A_{c,\si}=B^*_{A_{c,\si},c}\ge B^*_{A_c,c}=1$. 
Now, reasoning again similarly to \eqref{eq:ge}, one obtains the second line of \eqref{eq:trunc}. 

The proof of the strictness statement on \eqref{eq:trunc} is quite similar to that for \eqref{eq:c fixed}, because here as well one has $\be>0>\ga$ -- in either case, whether $\si^2\le A_c$ or $\si^2\ge A_c$.  
\end{proof}

\begin{proof}[Proof of Proposition~\ref{prop:trunc asymp}]\ 

Here the case $\si\downarrow0$ is quite staightforward. Indeed, then, by \eqref{eq:trunc}, one eventually has 
$
	L_{T;c,\si}-1=\E\exp\{c\,T\big(X_{\si^2,1}\big)\}-1
	=\frac{e^{-c\,\si^2}-1}{1+\si^2}\sim-c\,\si^2. 
$

As for the case $\si\to\infty$, the proof of \eqref{eq:T,c;si->infty} is quite similar to that of relation 
\eqref{eq:W,c;si->infty} in part (I) of Proposition~\ref{prop:win asymp}: replace all instances of $b^*_{a,c}$, $a_{c,\si}$, $b_{c,\si}$, $W$ by $B^*_{a,c}$, $A_{c,\si}$, $B_{c,\si}$, $T$, respectively, and also drop all instances of the factor $e^c$ in the numerators of the ratios in \eqref{eq:b*}, \eqref{eq:si^2}, and \eqref{eq:E W infty}. 
\end{proof}

\bibliographystyle{acm}
\bibliography{C:/Users/Iosif/Documents/mtu_home12-22-08/bib_files/citations}

\end{document}